\documentclass[12pt]{amsart}
\usepackage{amsmath}
\usepackage{amssymb}
\usepackage{latexsym}
\usepackage{amscd}
\usepackage{citesort}

\newdimen\AAdi%
\newbox\AAbo%
\def\AAk#1#2{\s_etbox\AAbo=\hbox{#2}\AAdi=\wd\AAbo\kern#1\AAdi{}}%
\def\AAr#1#2#3{\s_etbox\AAbo=\hbox{#2}\AAdi=\ht\AAbo\raise#1\AAdi\hbox{#3}}%
\font\tenmsb=msbm10 at 12pt \font\sevenmsb=msbm7 at 8pt
\font\fivemsb=msbm5 at 6pt
\newfam\msbfam
\textfont\msbfam=\tenmsb \scriptfont\msbfam=\sevenmsb
\scriptscriptfont\msbfam=\fivemsb
\def\Bbb#1{{\tenmsb\fam\msbfam#1}}
\newtheorem{thm}{Theorem}[section]
\newtheorem{lem}{Lemma}[section]

\newtheorem{rem}{Remark}[section]
\newtheorem{pro}{Proposition}[section]

\newcommand{\ba}{\begin{array}}
\newcommand{\ea}{\end{array}}

\newcommand{\Section}[2]{\setcounter{equation}{0}
\allowdisplaybreaks
\section[#1]{#2}}

\def\n{\nabla}

\def\ir#1{\mathbb R^{#1}}

\def\f#1#2{\frac{#1}{#2}}

\def\grs#1#2{\bold G_{#1,#2}}

\def\pd#1#2{\frac {\partial #1}{\partial #2}}

\def\td{\tilde}

\def\a{\alpha}
\def\be{\beta}

\def\de{\delta}
\def\De{\Delta}
\def\e{\eta}
\def\ep{\varepsilon}

\def\g{\gamma}

\def\la{\lambda}

\def\Om{\Omega}
\def\th{\theta}

\def\w{\wedge}

\def\U{\Bbb{U}}
\def\V{\Bbb{V}}
\def\ol{\overline}
\def\Hess{\mbox{Hess}}
\def\R{\Bbb{R}}

\begin{document}
\title
[Curvature estimates]{Curvature estimates for submanifolds with
prescribed Gauss image and  mean curvature}

\author
[Y.L. Xin]{Y. L. Xin}
\address
{Institute of Mathematics, Fudan University, Shanghai 200433, China
and Key Laboratory of Mathematics for Nonlinear Sciences (Fudan
University), Ministry of Education} \email{ylxin@fudan.edu.cn}
\thanks{The research was partially supported by NSFC}
\begin{abstract}
We study that the $n-$graphs defining by smooth map $f:\Om\subset
\ir{n}\to \ir{m}, m\ge 2,$ in $\ir{m+n}$ of the prescribed mean
curvature and the Gauss image. We derive the interior curvature
estimates
$$\sup_{D_R(x)}|B|^2\le\f{C}{R^2}$$
under the dimension limitations and the Gauss image restrictions. If
there is no dimension limitation we obtain
$$\sup_{D_R(x)}|B|^2\le C
R^{-a}\sup_{D_{2R}(x)}(2-\De_f)^{-\left(\f{3}{2}+\f{1}{s}\right)},
\qquad s=\min(m, n)$$ with $a<1$ under the condition
$$\De_f=\left[\text{det}\left(\de_{ij}+\sum_\a\pd{f^\a}{x^i}\pd{f^\a}{x^j}\right)\right]^{\f{1}{2}}<2.$$
If the image under the Gauss map is contained in a geodesic ball of
the radius  $\f{\sqrt{2}}{4}\pi$ in $\grs{n}{m}$ we also derive
corresponding estimates.
\end{abstract}

\renewcommand{\subjclassname}{%
  \textup{2000} Mathematics Subject Classification}
\subjclass{49Q05, 53A07, 53A10.}
\date{}
\maketitle

\Section{Introduction}{Introduction}
\medskip

There are many beautiful results on minimal hypersurfaces and we
have  a fairly profound understanding of the issue of minimal
hypersurfaces in many aspects, the issue of minimal submanifolds of
higher codimension seems more complicated. Lawson-Osserman's paper
revealed several important different phenomena  in higher
codimension \cite{l-o}.

For higher codimensional Bernstein problem Hildebrandt-Jost-Widman
\cite{h-j-w} gave us the following result.

\begin{thm}\label{C} Let $z^\a=f^\a(x), \a=1,\cdots,m,\;x=(x^1,\cdots,x^n)\in
\ir{n}$ be the $C^2$ solution to the system of minimal surface
equations. Let there exist $\be,$ where
$$\be<\cos^{-s}\left(\frac \pi{2\sqrt{s\,K}}\right),\quad K=\begin{cases} 1 \quad\text{if}\quad s=1\cr
         2 \quad \text{if} \quad s\ge 2\end{cases},\quad s=\min(m, n)$$
such that for any $x\in \ir{n}$,
$$\De_f=\left[\text{det}\left(\de_{ij}+\sum_\a\pd{f^\a}{x^i}\pd{f^\a}{x^j}\right)\right]^{\f{1}{2}}<\be,$$
then $f^1,\cdots,f^m$ are affine linear functions on $\ir{n},$
whose graph is an affine $n-$plane in $\ir{m+n}.$
\end{thm}

The theorem not only generalized Moser's  result \cite{m} to
higher codimension, but also introduced the Gauss image
assumption. The geometric meaning of the conditions in the above
theorem is that the image under the Gauss map lies in a closed
subset of an open geodesic ball of the radius
$\f{\sqrt{2}}{4}\pi$ in the Grassmannian manifold $\grs{n}{m}$.
It would be remind that the Gauss maps play an important role in
minimal surface theory. For general surfaces in $\ir{3}$ the
Gauss map and the mean curvature determine a simply connected
surface completely \cite{k}.

Later, in the author's joint work with J. Jost \cite{j-x}, the above
theorem has been improved that the number in the the theorem is $2$,
instead of $\cos^{-s}\left(\frac \pi{2\sqrt{s\,K}}\right)$, which is
independent of the dimension and the codimension. The key point is
to find larger geodesic convex set $B_{JX}\subset\grs{n}{m}$ which
contains the geodesic ball of radius $\f{\sqrt{2}}{4}\pi.$

In author's previous work, Schoen-Simons-Yau type curvature
estimates \cite{s-s-y} and Ecker-Huisken type curvature estimates
\cite{e-h} have be generalized to the flat normal bundle situation
\cite{x1}  \cite{s-w-x}.  Using some techniques in \cite{e-h1},
Fr\"ohlich-Winklmann \cite{f-w} derived interior  curvature
estimates for the flat normal bundle case and generalized our
results in \cite{s-w-x}.

Recently, we studied complete minimal submanifolds with codimension
$m\ge 2$ and curved normal bundle, but with the convex Gauss image.
Thus, we can construct auxiliary functions, which enable us to carry
out both Schoen-Simons-Yau type curvature estimates and
Ecker-Huisken type curvature estimates in \cite{x-y1} and
\cite{x-y2}. From the estimates several geometrical conclusions
follow, including the following Bernstein type theorems.

\begin{thm}\label{xy1}
Let $M$ be a complete minimal $n$-dimensional submanifold in
$\ir{n+m}$  with $n\leq 6$ and $m\ge 2$. If the Gauss image of $M$
is contained in an open geodesic ball of $\grs{n}{m}$ centered at
$P_0$ and of radius $\f{\sqrt{2}}{4}\pi$, then $M$ has to be an
affine linear subspace.

\end{thm}

\begin{thm}\label{xy2}
Let
 $M=(x,f(x))$ be an $n$-dimensional entire minimal graph given by $m$ functions
$f^\a(x^1,\cdots,x^n)$ with $n\le 4$ and $m\ge 2$. If
$$\De_f=\left[\text{det}\left(\de_{ij}+\sum_\a\pd{f^\a}{x^i}\pd{f^\a}{x^j}\right)\right]^{\f{1}{2}}<2,$$
then $f^\a$ has to be affine linear functions representing an affine
$n$-plane.
\end{thm}

\begin{thm}\label{xy3}
Let $M$ be a complete minimal $n$-dimensional submanifold in
$\ir{n+m}$. If the Gauss image of $M$ is contained in an open
geodesic ball of $\grs{n}{m}$ centered at $P_0$ and of radius
$\f{\sqrt{2}}{4}\pi$, and $\big(\f{\sqrt{2}}{4}\pi-\rho\circ
\g\big)^{-1}$ has growth
\begin{equation}
\big(\f{\sqrt{2}}{4}\pi-\rho\circ \g\big)^{-1}=o(R),
\end{equation}
where $\rho$ denotes the distance on $\grs{n}{m}$ from $P_0$ and $R$
is the Euclidean distance from any point in $M$. Then $M$ has to be
an affine linear subspace.
\end{thm}

\begin{thm}\label{xy4}
Let $M=(x,f(x))$ be an $n$-dimensional entire minimal graph given by
$m$ functions $f^\a(x^1,\cdots,x^n)$ with $m\geq 2$. If
$$\De_f=\left[\text{det}\left(\de_{ij}+\sum_\a\pd{f^\a}{x^i}\pd{f^\a}{x^j}\right)\right]^{\f{1}{2}}<2,$$
and
\begin{equation}
\left(2-\De_f\right)^{-1}=o(R^{\f{4}{3}}),
\end{equation}
where $R^2=|x|^2+|f|^2$. Then $f^\a$ has to be affine linear
functions and hence $M$ has to be an affine linear subspace.
\end{thm}

Let $\Om\subset\ir{n}$ be a domain and $f:\Om\to\ir{m}$ be a smooth
map whose graph is an $n-$submanifold $S$ in $\ir{m+n}$. We always
assume that $m\ge 2$ in this paper. On $S$ there is extrinsic
distance $r$, restriction to $S$ of Euclidean distance from $x\in
S$. Denote the closed ball of radius $R$ and centered at $x\in S$ by
$B_R(x)\subset\ir{m+n}$. Its restriction to $S$ is denoted by
$$D_R(x)=B_R(x)\cap S.$$ We also have the mean curvature $H$ and the
Gauss image restriction  $\g:M\to \V\subset\grs{n}{m}$, where $\V$
will be given in \S 3. We will  give interior estimates for the
squared norm of the second fundamental form $B$ of $S$ in $\ir{m+n}$
in terms of those geometric data. Since a complete submanifold in
Euclidean space with our Gauss image restrictions has to be a graph.
Those results can be viewed as generalizations of the above Theorems
\ref{xy1}-\ref{xy4}.

\begin{thm}\label{x1}
Let $S\subset\ir{m+n},\;  m\ge 2$ be a graph given by $f^1,\cdots,
f^m$ . Suppose $D_{2R}(x)\subset\subset S$. If one of the following
conditions is satisfied
\begin{enumerate}
\item $2\le n\le 6$ and the
image under the Gauss map is contained in an open geodesic ball of
radius $\f{\sqrt{2}}{4}\pi$;
\item $2\le n\le 5$ and
\begin{equation}\label{eq11}
\De_f=\left[\text{det}\left(\de_{ij}+\sum_\a\pd{f^\a}{x^i}\pd{f^\a}{x^j}\right)\right]^{\f{1}{2}}<2.
\end{equation}
\end{enumerate}
Then,
$$\sup_{D_R(x)}|B|^2\le\f{C}{R^2}$$
with the constant $C$ depending on $n, m, R\sup_{D_{2R}(x)}|H|,
R^2\sup_{D_{2R}(x)}(|\n H|+ \la_1)$ and $R^3\sup_{D_{2R}(x)}\la_2.$
\end{thm}

\begin{thm}\label{x2}
Let $S\subset\ir{m+n},\;  m\ge 2$ be a graph given by $f^1,\cdots,
f^m$. Suppose $D_{2R}(x)\subset\subset S$. If the image under the
Gauss map is contained in an open geodesic ball of radius
$\f{\sqrt{2}}{4}\pi$, then
\begin{equation}
\sup_{D_R(x)}|B|^2\le C
R^{-a}\sup_{D_{2R}(x)}\left(\f{\sqrt{2}}{4}\pi-\rho\circ\g\right)^{-2}.
\end{equation}
If
$$\De_f=\left[\text{det}\left(\de_{ij}+\sum_\a\pd{f^\a}{x^i}\pd{f^\a}{x^j}\right)\right]^{\f{1}{2}}<2,$$
then
\begin{equation}
\sup_{D_R(x)}|B|^2\le C
R^{-a}\sup_{D_{2R}(x)}(2-\De_f)^{-\left(\f{3}{2}+\f{1}{s}\right)},\qquad
s=\min(m, n).
\end{equation}
The above constant $C$ depends on $n, m, R\sup_{D_{2R}(x)}|H|,
R^2\sup_{D_{2R}(x)}(|\n H|+ \la_1)$ and $R^3\sup_{D_{2R}(x)}\la_2$
and the constant $a<1$.
\end{thm}

\begin{rem}
In the case of parallel mean curvature we can use Lemma \ref{mv}
with $g=0$ and the estimates in Theorem \ref{x2} could be improved
as
$$\sup_{D_R(x)}|B|^2\le C
R^{-2}\sup_{D_{2R}(x)}\left(\f{\sqrt{2}}{4}\pi-\rho\circ\g\right)^{-2}
$$
or $$\sup_{D_R(x)}|B|^2\le C
R^{-2}\sup_{D_{2R}(x)}(2-\De_f)^{-\f{3}{2}}$$ respectively, where
the constant $C$ depends  on $n, m, R\sup_{D_{2R}(x)}|H|.$
\end{rem}
\begin{rem}
$\la_1$ and $\la_2$ in Theorem \ref{x1} and Theorem \ref{x2} are
determined by mean curvature and Gauss image assumption. The precise
definitions are given by (\ref{lam1}) and (\ref{lam2}),
respectively.
\end{rem}

The paper will be arranged as follows. In \S 2 notations and basic
formulas will be given, especially, the Bochner-Simons type
inequality will be derived for general submanifolds in $\ir{m+n}$
with $m\ge 2$. In \S 3 we describe the two kinds Gauss image
restrictions which enable us to define auxiliary  functions. Those
are important in $L^p-$curvature estimates which is given in \S 4.
We will give Schoen-Simons-Yau type and Ecker-Huisken type estimates
in our general setting. In the final section we prove our main
results. It is done by using $L^p-$curvature estimates in \S 4 and
mean value inequality of Fr\"ohlich-Winklman in \cite{f-w}.

\Section{A Bochner-Simons type inequality}{A Bochner-Simons type
inequality}
\medskip

Let $M\to \ir{m+n}$ be an $n-$submanifold in $(m+n)-$dimensional
Euclidean space with the second fundamental form $B$ which can be
viewed as a cross-section of the vector bundle Hom($\odot^2TM, NM$)
over $M,$ where  $TM$ and $NM$ denote the tangent bundle and the
normal bundle  along $M$, respectively. A connection on
Hom($\odot^2TM, NM$) is induced from those of $TM$ and $NM$
naturally. We consider the  situation of  higher codimension $m\ge
2$ in this paper.

Taking the trace of $B$ gives the mean curvature vector $H$ of $M$
in $\ir{m+n}$,  a cross-section of the normal bundle.

To have the curvature estimates we need the Simons version of the
Bochner type formula for the squared norm of the second fundamental
form. It is done in \cite{jsi} for minimal submanifolds in an
arbitrary ambient Riemannian manifold. Now, for any submanifold in
Euclidean space, by the same calculation as in the paper \cite{jsi}
we have the following formula:
\begin{equation}\begin{split}
(\n^2 B)_{XY}
&=\n_X\n_Y H+ \left<B_{Xe_i},H\right>B_{Ye_i}-\left<B_{XY},B_{e_ie_j}\right>B_{e_ie_j}\\
&+2\left<B_{Xe_j},B_{Ye_i}\right>B_{e_ie_j}-\left<B_{Ye_i},B_{e_ie_j}\right>B_{Xe_j}
-\left<B_{Xe_i},B_{e_ie_j}\right>B_{Ye_j},\label{LB}
\end{split}\end{equation}
where $\n^2$ stands for the trace Laplacian operator, $\{e_i\}$
is a local tangential orthonormal frame field of $M$. Here and in
the sequel we use the summation convention. Then, we have (see
\cite{x2} for details)

\begin{pro}
\begin{equation}
\De|B|^2\ge 2\,|\n B|^2+2\,\left<\n_i\n_jH,B_{ij}\right>
       +\,2\left<B_{ij},H\right>\left<B_{ik},B_{jk}\right>
        -3|B|^4,\label{NLB'}
\end{equation}
where $\n_i$ denotes $\n_{e_i}$ and $B_{ij}=B_{e_i e_j}$.
\end{pro}

In (\ref{NLB'}), the terms involving the mean curvature can be
estimated as follows.

\begin{equation}
|\left<B_{ij}, H\right>\left<B_{ik}, B_{jk}\right>|\le |H||B|^3\le
\ep' |B|^4+\f{1}{\ep'}|H|^2|B|^2\label{H},\qquad
\ep'>0\end{equation} or
\begin{equation}
|\left<B_{ij}, H\right>\left<B_{ik}, B_{jk}\right>|\le
\sqrt{n}|B|^4.
\end{equation}

Define
\begin{equation}\label{lam2}
\la_2=\begin{cases} -\left(\f{\left<\n_i\n_jH,
B_{ij}\right>}{|B|}\right)^-,\quad\qquad\text{if}\qquad |B|>0,\\
0,\hskip1.6in\text{if}\qquad |B|=0,
\end{cases}
\end{equation}
where $(\cdots)^-$ denotes the negative part of the quantity.
Obviously,
$$\la_2\le |\n\n H|.$$

In order to use the  formula (\ref{NLB'}) we also need to
estimate $|\n B|^2$ in terms of $|\n|B||^2.$ Schoen-Simon-Yau
\cite{s-s-y} did such an estimate for codimension $m=1.$  The
following lemma is for any prescribed mean curvature $H$ and  any
codimension. In the case of $H=0$, the following estimates also
improve our previous estimates in \cite{x-y1}.

\begin{lem}For any real number $\ep>0$
\begin{equation}
|\n B|^2\ge\left(1+\f{2}{n+\ep}\right)|\n |B||^2-C(n,\ep)|\n
H|^2, \label{kato}
\end{equation}
where
$$C(n,\ep)=\f{2(n-1+\ep)}{\ep(n+\ep)}.$$
If $M$ has parallel mean curvature, then
\begin{equation}
|\n B|^2\ge\left(1+\f{2}{n}\right)|\n |B||^2.
\end{equation}
\end{lem}
\begin{proof} It is sufficient for us to prove the inequality at the points
where $|B|^2\neq 0$. Choose a local orthonormal tangent frame field
$\{e_1,\cdots, e_n\}$ and a local orthonormal normal frame field
$\{\nu_1,\cdots,\nu_m\}$ of $M$ near the considered point $x.$
Denote the shape operator $A^\a=A^{\nu_\a}.$  Then obviously
$|B|^2=\sum_\a |A^\a|^2$ and
$$\n |B|^2=\sum_\a \n |A^\a|^2.$$
Let
$$A^\a e_i=h_{\a ij}e_j,\qquad h_{\a ij}=h_{\a ji}.$$
By triangle inequality
$$\big|\n |B|^2\big|=\left|\sum_{\a}
\n |A^\a|^2\right|\leq\sum_\a \big|\n |A^\a|^2\big|. $$ Therefore,
\begin{equation}\label{ineq7}
\big|\n|B|\big|^2=\f{\big|\n |B|^2\big|^2}{4|B|^2}\leq \f{(\sum_\a
\big|\n |A^\a|^2\big|)^2}{4\sum_\a |A^\a|^2}.
\end{equation}

Since $|B|^2\neq 0$,  we can assume $|A^\a|^2>0$ for each $\a$
without loss of generality. Let $1\leq \g\leq m$ such that
\begin{equation*}
\f{\big|\n |A^\g|^2\big|^2}{|A^\g|^2}=\max_\a \Bigg\{\f{\big|\n
|A^\a|^2\big|^2}{|A^\a|^2}\Bigg\}<+\infty,
\end{equation*}
then from (\ref{ineq7}),
\begin{equation}\label{ineq7'}
\big|\n|B|\big|^2\leq \f{\big|\n |A^\g|^2\big|^2}{4|A^\g|^2}.
\end{equation}
$|A^\g|^2$ and $\n|A^\g|^2$ is independent of the choice of
$\{e_1,\cdots,e_n\}$, then without loss of generality we can assume
$h_{\g ij}=0$ whenever $i\neq j$. Then

\begin{eqnarray*}
\big|\n|A^\g|^2\big|^2 &=&4\sum_k \big(\sum_i h_{\g ii}h_{\g
iik}\big)^2 \\
&\le& 4\big(\sum_i h_{\g ii}^2\big)\big(\sum_{i,k}h_{\g
iik}^2\big)=4|A^\g|^2\sum_{i,k}h_{\g iik}^2\nonumber
\end{eqnarray*}
and from (\ref{ineq7'})
$$\big|\n |B|\big|^2\leq \sum_{i,k}h_{\g iik}^2.$$

Since
\begin{equation*}
\aligned \sum_i h_{\g iii}^2&=\sum_i(|\n_{e_i}H_\g|-\sum_{j\neq
i}h_{\g jji})^2\\
&=|\n H_\g|^2+ \sum_i (\sum_{j\neq i}h_{\g
jji})^2-2|\n H_\g|\sum_i\sum_{j\neq i}h_{\g jji}\\
&\le|\n H_\g|^2+(n-1)\sum_{j\neq i}h_{\g jji}^2+\f{n-1}{\ep}|\n
H_\g|^2+\f{\ep}{n-1}\sum_{j\neq
i}h_{\g jji}^2\\
&\le\left(1+\f{n-1}{\ep}\right)|\n H_\g|^2+(n-1+\ep)\sum_{j\neq
i}h_{\g jji}^2.
\endaligned
\end{equation*}
we obtain
\begin{equation}\begin{split}
\big|\n |B|\big|^2&\leq \sum_{i,k}h_{\g iik}^2
=\sum_{i\neq k} h_{\g iik}^2+\sum_i h_{\g iii}^2\\
&\le\left(1+\f{n-1}{\ep}\right)|\n H_\g|^2+(n+\ep)\sum_{j\neq
i}h_{\g jji}^2.\label{ineq8} \end{split}\end{equation}

On the other hand, a direct calculation shows
\begin{eqnarray}\label{ineq8'}
\big|\n|B|^2\big|^2&=&|2\sum_k \sum_{\a,i,j}h_{\a ij}h_{\a
ijk}e_k|^2
=4\sum_{\a,\be,i,j,s,t,k}h_{\a ij}h_{\a ijk}h_{\be st}h_{\be stk},\nonumber\\
|\n B|^2-\big|\n|B|\big|^2&=&|\n B|^2-\f{\big|\n|B|^2\big|^2}{4|B|^2}\nonumber\\
&=&\sum_{\a,i,j,k}h_{\a ijk}^2
-\f{\sum_{\a,\be,i,j,s,t,k}h_{\a ij}h_{\a ijk}h_{\be st}h_{\be stk}}{\sum_{\be,s,t} h_{\be st}^2}\nonumber\\
&=&\f{\sum_{\a,\be,i,j,s,t,k}(h_{\a ijk}h_{\be st}-h_{\be stk}h_{\a ij})^2}{2|B|^2}\nonumber\\
&\geq&\f{\sum_{\be,i\neq j,s,t,k}h_{\g ijk}^2h_{\be st}^2
+\sum_{\a,s\neq t,i,j,k}h_{\g stk}^2 h_{\a ij}^2}{2|B|^2}\nonumber\\
&=&\sum_{i\neq j,k}h_{\g ijk}^2\geq \sum_{i\neq k}(h_{\g iki}^2+h_{\g ikk}^2)\nonumber\\
&=&2\sum_{i\neq k}h_{\g iki}^2.
\end{eqnarray}

Noting (\ref{ineq8}) and (\ref{ineq8'}), we arrive at (\ref{kato}).
\end{proof}

Finally, we have
\begin{pro} In the case of $H\ne 0$
\begin{equation}\begin{split}
\De|B|^2\ge
2\left(1+\f{2}{n+\ep}\right)|\n|B||^2&-(3+2\ep')|B|^4\\
&-2\la_2|B|-2C(n,\ep)|\n H|^2-\f{2}{\ep'}|H|^2|B|^2\label{LB}
\end{split}\end{equation}
with $\ep$ and $\ep'$ and
\begin{equation}
\De|B|^2+(3+2\sqrt{n})|B|^4\ge -2\la_2|B|\label{LB'}.
\end{equation}
In the minimal case \begin{equation}\De|B|^2\ge
2\left(1+\f{2}{n}\right)|\n|B||^2-3|B|^4\label{LB1}\end{equation}
\end{pro}

\Section{Grassmannian manifolds and Gauss maps}{Grassmannian
manifolds and Gauss maps}
\medskip

Let $\ir{n+m}$ be an $(n+m)$-dimensional Euclidean space. All
oriented $n$-subspaces constitute the Grassmannian manifolds
$\grs{n}{m}$, which is an irreducible symmetric space of compact
type.

In the Grassmanian manifolds $\grs{n}{m}$ the sectional curvature of
the canonical metric varies in $[0, 2]$. The radius of  the largest
convex ball is $\f{\sqrt{2}}{4}\pi$.

We consider the two cases.

1. On an open geodesic ball
$B_{\f{\sqrt{2}}{4}\pi}(P_0)\subset\grs{n}{m}$ of radius
$\frac{\sqrt{2}}{4}\pi$ and centered at $P_0$.   Let
$$h=\cos(\sqrt{2}\rho)$$
be a positive function on $B_{\f{\sqrt{2}}{4}\pi}(P_0)$, where
$\rho$ is the distance function from $P_0$ in $\grs{n}{m}$. Then,
the Hessian comparison theorem gives
\begin{equation}
\text{Hess}(h)=h'\text{Hess}(\rho)+h''d\rho\otimes d\rho \le-2 h\
g\end{equation} with the metric tensor $g$ on $\grs{n}{m}$.

Let
$$h_1=\sec^2(\sqrt{2}\rho),$$
where $\rho$ is the distance function from $P_0$ in $\grs{n}{m}$.
We then have
$$\aligned
\text{Hess}(h_1)&=h_1'\text{Hess}(\rho)+h_1''d\rho\otimes d\rho\\
&\ge  4 h_1\ g+\f{3}{2}h_1^{-1}dh_1\otimes dh_1\endaligned $$

2. For $P_0\in \grs{n}{m}$, which is expressed by a unit
$n-$vector $\ep_1\w\cdots\w\ep_n$. For any $P\in\grs{n}{m}$,
expressed by an $n-$vector $e_1\w\cdots\w e_n$, we define an
important function on $\grs{n}{m}$
$$w\mathop{=}\limits^{def.}\left<P, P_0\right>
=\left<e_1\wedge\cdots\wedge e_n, \ep_1\wedge\cdots\wedge
\ep_n\right>=\det W,$$ where $W=(\left<e_i, \ep_j\right>).$ The
Jordan angles between $P$ and $P_0$ are defined by
$$\th_\a=\arccos(\la_\a),$$
where $\la_\a\geq 0$ and $\la_\a^2$ are the eigenvalues of the
symmetric matrix $W^TW$ The distance between $P_0$ and $P$ is
$$d(P_0, P)=\sqrt{\sum\th_\a^2}.$$

Denote
\begin{equation*}
\U=\{P\in \grs{n}{m}:w(P)>0\}.
\end{equation*}
On $\U$ we can define
$$v=w^{-1}=\prod_\a\sec\th_\a.$$

Define
$$B_{JX}(P_0)=\big\{P\in \U:\mbox{ sum of any two Jordan angles} \mbox{between }P\mbox{ and }P_0<\f{\pi}{2}\big\}.$$
This is a geodesic convex set, larger than the geodesic ball of
radius $\frac{\sqrt{2}}{4}\pi$ and centered at $P_0$. This  was
found in a  previous work of Jost-Xin \cite{j-x}. For any real
number $a$ let $\V_a=\{P\in\grs{n}{m},\quad v(P)<a\}.$ From
(\cite{j-x}, Theorem 3.2) we know that
$$\V_2\subset B_{JX}\qquad \text{and}\qquad
\ol{\V}_2\cap\ol{B}_{JX}\neq\emptyset.$$

\begin{thm}\cite{x-y2}

$v$ is a convex function on $B_{JX}(P_0)\subset \U\subset
\grs{n}{m}$,  and
\begin{equation}\label{es4}
\Hess(v)\geq
v(2-v)g+\Big(\f{v-1}{pv(v^{\f{2}{p}}-1)}+\f{p+1}{pv}\Big)dv\otimes
dv
\end{equation}
on $\ol{\V}_2$, where $g$ is the metric tensor on $\grs{n}{m}$ and
$p=min(n,m)$.
\end{thm}

Let
\begin{equation}
h=v^{-k}(2-v)^k
\end{equation}
define a positive function on $\V_2$, where $k=\f{3}{4}+\f{1}{2s}$
and $s=\min(m, n).$ From (\ref{es4}) we have (see (4.4) in
\cite{x-y2})
\begin{equation} \Hess(h)\le
-\left(\f{3}{2}+\f{1}{s}\right)h\ g,
\end{equation}
where $g$ is the metric tensor on $\grs{n}{m}.$

Let
\begin{equation}\label{h2}
h_1=h^{-2},
\end{equation}
then
\begin{equation}\begin{split}
\Hess(h_1)&=h_1'\Hess(h)+h_1''dh\otimes dh\\
&\ge \left(3+\f{2}{s}\right)h_1\ g+\f{3}{2}h_1^{-1}dh_1\otimes
dh_1,
\end{split}
\end{equation}
where $g$ is the metric tensor on $\grs{n}{m}.$

In each of the above cases we have a positive functions $h$ and
$h_1$ defined on an open subset on $\V\subset\grs{m}{n}$
satisfying
\begin{equation}\label{lh}
\Hess (h)\le -\la h g,
\end{equation}
where $\la=2$ in the first case and $\la>\f{3}{2}$ in the second
case and
\begin{equation}\label{lh1'}
\Hess (h_1)\ge \mu h_1 g+\f{3}{2}h_1^{-1}dh_1\otimes dh_1,
\end{equation}
where $\mu=4$ in the first case and $\mu>3$ in the second case.

For  $n$-dimensional submanifold  $M$ in $\R^{n+m}$. The Gauss map $
\gamma : M \to \grs{n}{m} $ is defined by
$$
 \g (x) = T_x M \in \grs{n}{m}
$$
via the parallel translation in $ \ir{m+n} $ for arbitrary $ x \in
M$. If $m=1$, the image of the Gauss map is the unit sphere. This is
just the hypersurface situation. Otherwise, the image of the Gauss
map is a Grassmannian manifold.

The energy density of the Gauss map (see \cite{x} Chap.3, \S 3.1) is
$$e(\g)=\f{1}{2}\left<\g_*e_i,\g_*e_i\right>=\f{1}{2}|B|^2.$$

We assume that the image of  $M$ under the Gauss map is contained in
$\V\subset\grs{m}{n}.$ Thus, we have the function $\td h=h\circ \g$
and $\td h_1=h_1\circ\g$ defined on $M$. We denote $h$ for $\td h$
and $h_1$ for $\td h_1$ in the sequel for simplicity.  Define
\begin{equation}\label{lam1}
(dh(\tau(\g))^+=\la_1h,
\end{equation} where $\tau(\g)$ is the
tension field of the Gauss map, which is zero when $M$ has parallel
mean curvature by Ruh-Vilms theorem \cite{r-v}.

From (\ref{lh}) and the composition formula we have
\begin{equation}\begin{split}
\De h&\le -\la |B|^2h+dh(\tau(\g))\\
&\qquad\le-\la|B|^2h+\la_1h.
\end{split}\end{equation}
We then obtain
$$\la|B|^2\phi^2\le -h^{-1}\phi^2\De h+\la_1\phi^2$$
and
$$\la\int_M|B|^2\phi^2*1\le-\int_M h^{-1}\phi^2\De h*1+\int_M \la_1\phi^2*1.$$
Since \begin{equation}\begin{split} -\int_M &h^{-1}\phi^2\De h*1
=-\int_M\n(h^{-1}\phi^2\n h)*1+\int_M\n(h^{-1}\phi^2)\n h*1\\
&=\int_M(\n h^{-1})(\n h)\phi^2*1+\int_Mh^{-1}\n h\n \phi^2*1\\
&=-\int_M h^{-2}|\n h|^2\phi^2*1+2\int_M h^{-1}\phi\n h\cdot\n\phi*1\\
&\le -\int_M h^{-2}|\n h|^2\phi^2*1 +\int_M h^{-2}\phi^2|\n
h|^2*1 +\int_M |\n\phi|^2*1\\
&=\int |\n\phi|^2*1,
\end{split}\end{equation}
we obtain
\begin{equation}\la\int_M|B|^2\phi^2\le\int_M|\n\phi|^2*1+\int_M
\la_1\phi^2*1\label{lp1} \end{equation} for arbitrary function
$\phi$ with compact support $D\subset M$.

Define
$$\mu_1=-h_1^{-1}(dh_1(\tau(\g))^-.$$
Then from (\ref{lh1'}) and the composition formula we have
\begin{equation}\label{dh2}
\De h_1\ge \mu h_1|B|^2+\f{3}{2} h_1^{-1}|\n h_1|^2-\mu_1 h_1,
\end{equation}
where $\mu>3$.

In the graphic situation the $v$-function on $\grs{n}{m}$ composed
with the Gauss map $\g$ is just
$$\De_f=\left[\text{det}\left(\de_{ij}+\sum_\a\pd{f^\a}{x^i}\pd{f^\a}{x^j}\right)\right]^{\f{1}{2}},$$
which is the volume element of our graph. It follows that
$\text{vol}(D_R(x))\le 2R^n$ for the second case. As for the first
case
$$\De_f\le\left(\sec\left(\f{\sqrt{2}}{4}\pi\right)\right)^s$$
with $s=\min(m, n)$. So in each cases we have
$$\text{vol}(D_R(x))\le C\ R^n$$
with the constant $C$ depending on $n$ and $m$.

\Section{$L^p-$Curvature estimates}{$L^p-$Curvature estimates}

\medskip

Replacing $\phi$ by $|B|^{1+q}\phi$ in (\ref{lp1}) gives
\begin{equation}\label{lp2}\begin{split}
\int_M |B|^{4+2q}\phi^2*1&\le \la^{-1}\int_M \big|\n
(|B|^{1+q}\phi)\big|^2*1
+\la^{-1}\int_M\la_1|B|^{2+2q}\phi^2\\
&=\la^{-1}(1+q)^2\int_M |B|^{2q}\big|\n|B|\big|^2\phi^2*1\\
&\qquad+\la^{-1}\int_M |B|^{2+2q}(|\n \phi|^2+\la_1\phi^2)*1\\
&\qquad+2\la^{-1}(1+q)\int_M |B|^{1+2q}\n|B|\cdot \phi\n\phi
*1.\end{split}
\end{equation}

Using Bochner type formula (\ref{LB}), which is equivalent to
\begin{equation}\label{lp3}\begin{split}
\f{2}{n+\ep}|\n|B||^2&\le |B|\De |B|+\f{(3+2\ep')}{2}|B|^4\\&\qquad
+\la_2|B|+\f{C(n,\ep)}{2}|\n H|^2+\f{C(n)}{2\ep'}|H|^2|B|^2.
\end{split}\end{equation}

Multiplying $|B|^{2q}\phi^2$ with both sides of (\ref{lp3}) and
integrating by parts, we have
\begin{equation}\label{lp4}\aligned
\f{2}{n+\ep}&\int_M |B|^{2q}\big|\n |B|\big|^2\phi^2*1\\
&\qquad\le-(1+2q)\int_M |B|^{2q}\big|\n|B|\big|^2\phi^2*1\\
&\qquad\qquad-2\int_M |B|^{1+2q}\n|B|\cdot \phi\n\phi*1
+\f{3+\ep'}{2}\int_M |B|^{4+2q}\phi^2 *1\\
&\qquad\qquad+\int_M\la_2|B|^{1+2q}\phi^2*1+\f{C(n)}{2\ep'}\int_M|H|^2|B|^{2+2q}\phi^2*1\\
&\qquad\qquad+\f{C(n,\ep)}{2}\int_M|\n H|^2|B|^{2q}\phi^2*1.
\endaligned
\end{equation}

By multiplying $\f{3+\ep'}{2}$ with both sides of (\ref{lp2}) and
then adding up both sides of it and (\ref{lp4}), we have
\begin{equation}\label{lp5}\aligned
&\big(\f{2}{n+\ep}+1+2q-\f{3+\ep'}{2}\la^{-1}(1+q)^2\big)\int_M |B|^{2q}\big|\n |B|\big|^2\phi^2*1\\
&\qquad\leq\f{3+\ep'}{2}\la^{-1}\int_M
|B|^{2+2q}|(\n\phi|^2+\la_1\phi^2)*1\\
&\qquad\qquad+\big((3+\ep')\la^{-1}(1+q)-2\big)\int_M
|B|^{1+2q}\n|B|\cdot
\phi\n\phi *1\\
&\hskip1in+\int_M\la_2|B|^{1+2q}\phi^2*1+\f{C(n)}{2\ep'}\int_M|H|^2|B|^{2+2q}\phi^2*1\\
&\hskip1.2in+\f{C(n,\ep)}{2}\int_M|\n H|^2|B|^{2q}\phi^2*1.
\endaligned\end{equation}
By using Young's inequality and letting $\ep=\ep'$, (\ref{lp5})
becomes
\begin{equation}\label{lp6}
\aligned
&\big(\f{2}{n+\ep}+1+2q-\f{3+\ep}{2}\la^{-1}(1+q)^2-\ep\big)\int_M |B|^{2q}\big|\n |B|\big|^2\phi^2*1\\
&\hskip1in\le C_1(\ep,\la,q,n)\int_M
|B|^{2+2q}(|\n\phi|^2+\la_1\phi^2+|H|^2\phi^2)*1\\
&\hskip1.2in + C_1(\ep,\la,q,n)\int_M\la_2|B|^{1+2q}\phi^2\\
&\hskip1.5in + C_1(\ep,\la,q,n)\int_M|\n
H|^2|B|^{2q}\phi^2.\endaligned
\end{equation}

If
\begin{equation}\label{co3}
\la>\f{3}{2}\big(1-\f{2}{n}\big),
\end{equation}
then
\begin{equation*}
\f{2}{n}+1+2q-\f{3}{2}\la^{-1}(1+q)^2>0
\end{equation*}
whenever
\begin{equation}\label{lp6'}
q\in
\Big[0,-1+\f{2}{3}\la+\f{1}{3}\sqrt{4\la^2-6\big(1-\f{2}{n}\big)\la}\
\Big).
\end{equation}
Thus we can choose $\ep$ sufficiently small, such that
\begin{equation}\label{lp7}\begin{split}
\int_M |B|^{2q}\big|\n |B|\big|^2\phi^2*1&\le C_2 \int_M
|B|^{2+2q}|(\n\phi|^2+\la_1\phi^2+|H|^2\phi^2)*1\\
&+C_2\int_M\la_2|B|^{1+2q}\phi^2*1+C_2\int_M|\n
H|^2|B|^{2q}\phi^2\end{split}
\end{equation}
where $C_2$ only depends on $n$, $\la$ and $q$.

Combining with (\ref{lp2}) and (\ref{lp7}), we can derive
\begin{equation}\label{lp8}\begin{split}
\int_M |B|^{4+2q}\phi^2*1&\le C_3(n,\la,q)\int_M |B|^{2+2q}|(\n
\phi|^2+\la_1\phi^2+|H|^2\phi^2)*1\\
&+C_3\int_M\la_2|B|^{1+2q}\phi^2*1+C_3\int_M|\n
H|^2|B|^{2q}\phi^2\end{split}\end{equation} by  using Young's
inequality again.

Replacing $\phi$ by $\phi^{q+2}$ in (\ref{lp7}) yields
\begin{equation}\label{lp9}\begin{split}
\int_M &|B|^{4+2q}\phi^{4+2q}*1\le C\int_M
|B|^{2+2q}\phi^{2+2q}|(\n
\phi|^2+\la_1\phi^2+|H|^2\phi^2)*1\\
&+C\int_M|B|^{1+2q}\phi^{1+2q}\la_2\phi^3*1+C\int_M|B|^{2q}\phi^{2q}\n
|H|^2\phi^4*1,\end{split}\end{equation} in what follows $C$ may be
different in different expressions which depending on  $n$, $\la$
and $q$.

By using Young's inequality, namely for any positive real number
$\a,\,a,\,b,\,s,\,t\,$ with $\frac 1s+\frac 1t=1$
$$\frac {\a^sa^s}s+\frac {\a^{-t}\,b^t}t\ge ab,$$
we have
\begin{equation*}\begin{split}
C|B|^{2+2q}\phi^{2+2q}|\n\phi|^2&\le
\ep|B|^{4+2q}\phi^{4+2q}+\a_0|\n\phi|^{4+2q},\\
C|B|^{2+2q}|H|^2&\le \ep|B|^{4+2q}+\a_1|H|^{4+2q},\\
C|B|^{2+2q}\la_1&\le\ep|B|^{4+2q}+\a_2\la_1^{2+q},\\ C|B|^{2q}|\n
H|^2&\le \ep|B|^{4+2q}+\a_3|\n H|^{2+q}, \\
C|B|^{1+2q}\la_2&\le \ep|B|^{4+2q}+\a_4\la_2^{\f{4+2q}{3}}.
\end{split}\end{equation*}
Finally, (\ref{lp9}) becomes
\begin{equation}\label{lp10}\begin{split}
&\int_M|B|^{4+2q}\phi^{4+2q}*1\\
&\le C\int_M|\n\phi|^{4+2q}*1\\
&\qquad+C\int_M\left(|H|^{4+2q}+|\n
H|^{2+2q}+\la_1^{2+2q}+\la_2^{\f{4+2q}{3}}\right)\phi^{4+2q}*1.
\end{split}
\end{equation}

Let $r$ be a function on $M$ with $|\n r|\le 1$. For any $R\in [0,
R_0]$, where $R_0=\sup_Mr$, suppose
$$M_R=\{x\in M,\quad r\le R\}$$
is compact.

(\ref{lp6'}) and (\ref{lp10}) enable us to prove the following
results.

\begin{thm}\label{t1}
Let $M$ be an $n$-dimensional submanifolds of $\ir{n+m}$ with mean
curvature $H$. If the Gauss image of $M_{2R}$ is contained in an
open geodesic ball of radius $\f{\sqrt{2}}{4}\pi$ in $\grs{n}{m}$,
then we have the $L^p$-estimate
\begin{equation}\label{Lp15}
\big\||B|\big\|_{L^p(M_R)}\leq C\
R^{-1}\text{Vol}(M_{2R})^{\f{1}{p}}
\end{equation}
for
$$p\in \left[4,4+\f{2}{3}+\f{4}{3}\sqrt{1+\f{6}{n}} \right),$$
where $C$ is depending on $n, R\sup_{M_{2R}}|H|,
R^2\sup_{M_{2R}}(|\n H|+ \la_1)$ and $R^3\sup_{M_{2R}}\la_2.$
\end{thm}

\begin{proof} Take $\phi\in C_c^\infty (M_R)$ to be the
standard cut-off function such that $\phi\equiv 1$ in $M_R$ and $|\n
\phi|\le C R^{-1}$; then (\ref{lp10}) yields
$$\int_{M_R}|B|^p*1\le C\ R^{-p}\text{Vol}(M_{2R}),$$
where $p=4+2q$. Thus the conclusion immediately follows from
(\ref{lp10}).
\end{proof}

\begin{thm}\label{t2}
Let $M$ be an $n$-dimensional submanifolds of $\ir{n+m}$ with the
mean curvature $H$. If the Gauss image of $M_{2R}$ is contained in
$\{P\in \U\subset \grs{n}{m}:v(P)<2\}$, then we have the estimate
\begin{equation}\label{lp11}
\big\||B|\big\|_{L^p(M_R)}\le C\ R^{-1}\text{Vol}(M_{2R})^{\f{1}{p}}
\end{equation}
for
$$p\in \left[4,4+\sqrt{\f{8}{n}}\right),$$
where $C$ is depending on $n, R\sup_{M_{2R}}|H|,
R^2\sup_{M_{2R}}(|\n H|+\la_1)$ and $R^3\sup_{M_{2R}}\la_2.$
\end{thm}

We now study the $L^p-$curvature estimates in terms of $h_1.$  From
(\ref{LB}) and (\ref{dh2}) we compute
\begin{equation}\begin{split}
\De(|B|^{2p}h_1^q)&\ge (\mu
q-3p-2\ep'p)|B|^{2p+2}h_1^q\\
&\qquad+2p\left(2p-1+\f{2}{n+\ep}\right)|B|^{2p-2}h_1^q|\n |B||^2\\
&+q(q+\f{1}{2})|B|^{2p}h_1^{q-2}|\n h_1|^2+4pq|B|^{2p-1}\n
|B| \cdot h_1^{q-1}\n h_1\\
&\qquad-\f{2p}{\ep'}|H|^2|B|^{2p}h_1^q-q\mu_1|B|^{2p}h_1^q\\
&\qquad-2p\la_2|B|^{2p-1}h_1^q-2pC(n,\ep)|B|^{2p-2}h_1^q|\n H|^2.
\end{split}\end{equation}
Using Young's inequality, when $p\ge
\f{1}{2}-\f{1}{n}+\left(1-\f{2}{n}\right)q$, we obtain
\begin{equation}\label{sh}\begin{split}
\De(|B|^{2p}h_1^q)&\ge (\mu q-3p-2\ep'p)|B|^{2p+2}h_1^q
-\f{2p}{\ep'}|H|^2|B|^{2p }h_1^q\\&\qquad
-q\mu_1|B|^{2p}h_1^q-2p\la_2|B|^{2p-1}h_1^q-2pC(n,\ep)|B|^{2p-2}h_1^q|\n
H|^2.\end{split}\end{equation} In particular, we have when $p\ge
n-1$
\begin{equation}\begin{split}
\De(|B|^{p-1}h_1^{\f{p}{2}}&\ge
\f{3}{2}|B|^{p+1}h_1^{\f{p}{2}}-\f{p-1}{\ep'}|H|^2|B|^{p-1
}h_1^{\f{p}{2}}-\f{p}{2}\mu_1|B|^{p-1}h_1^{\f{p}{2}}\\
&\qquad-(p-1)\la_2|B|^{p-2}h_1^{\f{p}{2}}-(p-1)C(n,\ep)|B|^{p-3}h_1^{\f{p}{2}}|\n
H|^2.\end{split}\end{equation}

Multiplying $|B|^{p-1}h_1^{\f{p}{2}}\e^{2p}$, integrating by
parts and using Young's inequality lead
\begin{equation}\begin{split}
\int_M&|B|^{2p}h_1^p\e^{2p}*1\le\f{2}{3}p^2\int_M|B|^{2p-2}h_1^p\e^{2p-2}|\n\e|^2*1\\
&+\f{2(p-1)}{3\ep'}\int_M|H|^2|B|^{2p-2}h_1^p\e^{2p}*1
+\f{p}{3}\int_M\mu_1|B|^{2p-2}h_1^p\e^{2p}*1\\
&\qquad\quad+\f{2}{3}(p-1)\int_M\la_2|B|^{2p-3}h_1^p\e^{2p}*1\\
&\qquad\qquad+\f{2}{3}(p-1)C(n,\ep)\int_M|B|^{2p-4}h_1^p|\n
H|^2\e^{2p}*1,
\end{split}\end{equation}
where $\e$ is a smooth function with compact support. By using
Young's inequality again, we obtain
\begin{equation}\label{lp14}\begin{split}
\int_M|B|^{2p}h_1^p\e^{2p}*1&\le C\int_Mh_1^p|\n\e|^{2p}*1\\
&+C\int_M\left(|H|^{2p}+\mu_1^p+\la_2^{\f{2}{3}p}+|\n
H|^p\right)h_1^p\e^{2p}*1,
\end{split}\end{equation}
where $C$ is a constant depending on $p$ and $n$. Take $\eta\in
C_c^\infty (D_{2R}(x))$ to be the standard cut-off function such
that $\eta\equiv 1$ in $D_R$ and  $|\n \eta|\le C R^{-1}$; then from
(\ref{lp14}) we have the following estimate.

\begin{thm}\label{t3}
Let $M$ be an $n$-dimensional minimal submanifold of $\ir{n+m}$.
If there exists a positive function $h_1$ on $M$ satisfying
(\ref{dh2}), then
\begin{equation}\label{lp15}
\int_M|B|^{2p}h_1^p\e^{2p}*1\le CR^{n-2p}\sup_{D_{2R}}h_1^p
\end{equation}
where $C$ depends on $p, n, m, R\sup_{D_{2R}}h_1^p|H|,
R^2\sup_{D_{2R}}(\mu_1+|\n H|), R^3\sup_{D_{2R}}\la_2$.
\end{thm}

\Section{Proof of the main results}{Proof of the main results}

\medskip

From $L^p-$estimates to pointwise estimates we need the following
mean value inequality in \cite{f-w}:

\begin{lem}\label{mv}
Let S be an $n-$graph in $\ir{m+n}.$ Suppose that $u$ is a
nonnegative solution of
\begin{equation}
\De u+Q u\ge g \qquad\text{on\;  S}
\end{equation}
where $Q\in L^{\f{q'}{2}}(S)$ and $g\in L^{\f{p'}{2}}(S)$ with $q',
p'>n$. If $D_{2R}(x)\subset\subset S$, then we have the estimates
\begin{equation}
\sup_{D_R(x)} u\le
C\left(R^{-\f{n}{2}}||u||_{L^2(D_{2R}(x))}+k(R)\right),
\end{equation}
where
\begin{equation}
k(R)=R^{2(1-\f{n}{p'})}||g||_{L^{\f{p'}{2}}}(D_{2R}(x)),
\end{equation}
the constants $C$ depending on $n,\ p',\ q',\
R^{2(1-\f{n}{q'})}||Q||_{L^{\f{q'}{2}}(D_{2R}(x))},\
R\sup_{D_{2R}(x)}|H|$ and $R^{-n}\text{Vol}(D_{2R}(x)).$
\end{lem}

\noindent {\bf Proof of Theorem \ref{x1}}
\medskip

 Now, noting (\ref{LB'}) we use Lemma \ref{mv} with
 $$u=|B|^2, \; Q=(3+2\sqrt{n})|B|^2,\;
 g=-2\la_2|B|\quad\text{and}\quad p'=2q'=2p.$$
 Using Theorem \ref{t1} and Theorem \ref{t2} have
 \begin{equation}\begin{split}
 ||Q||_{L^{\f{p}{2}}(D_{2R}(x))}&=(3+2\sqrt{n})\left(\int_{D_{2R}(x)}|B|^p*1\right)^{\f{2}{p}}\\
 &\le C\ R^{-2}\text{Vol}(D_{2R}(x))^{\f{2}{p}},
 \end{split}
 \end{equation}

 $$R^{2\left(1-\f{n}{p}\right)}||Q||_{L^{\f{p}{2}}(D_{2R}(x))}\le C\
 R^{-\f{2n}{p}}\text{Vol}(D_{2R}(x))^{\f{2}{p}},$$

 \begin{equation}\begin{split}
 ||g||_{L^{\f{p'}{2}}(D_{2R}(x))}=||g||_{L^p(D_{2R}(x))}&\le C\
 R^{-3}\left(\int_{D_{2R}(x)}|B|^p*1\right)^{\f{1}{p}}\\
 &\le C\ R^{-4}\text{Vol}(D_{2R}(x))^{\f{1}{p}},
\end{split}
 \end{equation}
 $$k(R)=R^{2\left(1-\f{n}{2p}\right)}||g||_{L^p(D_{2R}(x))}\le C\
 R^{-2}R^{-\f{n}{p}}\text{Vol}(D_{2R}(x))^{\f{1}{p}},$$
 $$||u||_{L^2(D_{2R}(x))}=\left(\int_{D_{2R}(x)}|B|^4*1\right)^{\f{1}{2}}\le\
 C R^{-2}\text{Vol}(D_{2R}(x))^{\f{1}{2}},$$
 $$R^{-\f{n}{2}}||u||_{L^2(D_{2R}(x))}\le C\ R^{-2}R^{-\f{n}{2}}\text{Vol}(D_{2R}(x))^{\f{1}{2}},$$
 Hence,
 $$\sup_{D_R(x)}u\le C\ R^{-2-\f{n}{2}}\text{Vol}(D_{2R}(x))^{\f{1}{2}}.$$
In \S 3 we have shown the volume growth under our Gauss image
assumption. We then finish the proof Theorem \ref{x1}.

\newpage
\noindent {\bf Proof of Theorem \ref{x2}}
\medskip

From (\ref{sh}) we also have (in case of $p\ge n-1$)
\begin{equation}\label{sh1}\begin{split}
\De (|B|^{2p}h_1^p)&\ge -\f{2p}{\ep'}|H|^2|B|^{2p}
h_1^p-p\mu_1|B|^{2p}h_1^p\\
&\qquad-2p\la_2|B|^{2p-1}h_1^p-2pC(n,\ep)|B|^{2p-2}h_1^p|\n
H|^2.\end{split}\end{equation}

By Young's inequality
\begin{equation}\begin{split}
|B|^{2p-1}h_1^p\la_2&=|B|^{2p-1}h_1^{p-\f{1}{2}}\la_2^{\f{2p-1}{3p}}h_1^{\f{1}{2}}\la_2^{\f{p+1}{3p}}\\
&\le C_1|B|^{2p}h_1^p\la_2^{\f{2}{3}}+C_2h^p\la_2^{\f{2(p+1)}{3}},\\
|B|^{2p-2}h_1^p|\n H|^2&=|B|^{2p-2}h_1^{p-1}|\n H|h_1|\n H|\\
&\le C_3|B|^{2p}h_1^p|\n H|^{\f{p}{p-1}}+C_4h^p|\n H|^p.
\end{split}\end{equation}
Then (\ref{sh1}) becomes
\begin{equation}\begin{split}
\De (|B|^{2p}h_1^p)&+ C(n,p)(|H|^2+\mu_1+\la_2^{\f{2}{3}}+|\n
H|^{\f{p}{p-1}})|B|^{2p}h_1^p\\
&\ge -C(n,p)(\la_2^{\f{2(p+1)}{3}}h^p+|\n H|^ph^p) .
\end{split}\end{equation}

We use Lemma \ref{mv} with
$$u=|B|^{2p}h_1^p,\quad Q=C(n, p)(|H|^2+\mu_1+\la_2^{\f{2}{3}}+|\n
H|^{\f{p}{p-1}}),$$ $$g=-C(n,p)(\la_2^{\f{2(p+1)}{3}}h^p+|\n
H|^ph^p)\quad \text{and}\quad  p'=2q'=2p>2n.$$ Since
\begin{equation}\begin{split}
\left(\int_{D_{2R}(x)}|H|^p*1\right)^{\f{2}{p}}&\le C
R^{\f{2n}{p}-2},\qquad
\left(\int_{D_{2R}(x)}\mu_1^{\f{p}{2}}*1\right)^{\f{2}{p}}\le C
R^{\f{2n}{p}-2},\\
\left(\int_{D_{2R}(x)}\la_2^{\f{p}{3}}*1\right)^{\f{2}{p}}&\le C
R^{\f{2n}{p}-2},\qquad \left(\int_{D_{2R}(x)}|\n
H|^{\f{p^2}{2(p-1)}}*1\right)^{\f{2}{p}}\le C
R^{\f{2n}{p}-\f{2p}{p-1}},
\end{split}\end{equation}we obtain
$$R^{2(1-\f{n}{p})}||Q||_{L^{\f{p}{2}}(D_{2R}(x))}\le C,$$
where the constant $C$ depending on $n,m, p, R\sup_{D_{2R}(x)}|H|,
R^2\sup_{D_{2R}(x)}(|\n H|+ \mu_1)$ and $R^3\sup_{D_{2R}(x)}\la_2.$
We also have
\begin{equation}\begin{split}
&\left(\int_{D_{2R}(x)}(\la_2^{\f{2(p+1)}{3}}h_1^p)^p*1\right)^{\f{1}{p}}\le
C R^{\f{n}{p}-p-1}\sup_{D_{2R}(x)}h_1^p\\
&R^{2(1-\f{n}{2p})}||\la_2^{\f{2(p+1)}{3}}h_1^p||_{L^p(D_{2R}(x))}\le
C R^{-p+1}\sup_{D_{2R}(x)}h_1^p
\end{split}\end{equation}
and \begin{equation} \left(\int_{D_{2R}(x)}(|\n
H|^ph_1^p)^p*1\right)^{\f{1}{p}}\le C
R^{\f{n}{p}-2p}\sup_{D_{2R}(x)}h_1^p,
\end{equation}
\begin{equation}
R^{2\left(1-\f{n}{2p}\right)}|||\n H|^ph_1^p||_{L^p(D_{2R}(x))}\le
CR^{-2p+2}\sup_{D_{2R}(x)}h_1^p .\end{equation} It follows that
$$k(R)\le C R^{-p+1}\sup_{D_{2R}(x)}h_1^p.$$
From Theorem \ref{t3} we have
$$\left(\int_{D_{2r}(x)}|B|^{4p}h_1^{2p}*1\right)^{\f{1}{2}}\le C
R^{\f{n}{2}-2p}\sup_{D_{2R}(x)}h_1^p$$ and
$$R^{-\f{n}{2}}|||B|^{2p}h_1^p||_{L^2(D_{2R}(x))}\le
R^{-2p}\sup_{D_{2R}(x)}h_1^p.$$

In the case of Gauss image is contained in a geodesic ball of radius
$\f{\sqrt{2}}{4}$
$$h_1=\sec^2(\sqrt{2}\rho).$$
It is easily seen that
$$\sec(\sqrt{2}\rho)\le C \left(\f{\sqrt{2}}{4}-\rho\right)^{-1}$$
for the constant $C>0$.

In the case of $\De_f<2$
$$h_1=\left(\f{v}{2-v}\right)^{-\left(\f{3}{2}+\f{1}{s}\right)}\le C(2-v)^{-\left(\f{3}{2}+\f{1}{s}\right)}$$
with the constant $C>0$. We then finish the proof of Theorem
\ref{x2}.

\bibliographystyle{amsplain}

\begin{thebibliography}{10}

\bibitem{e-h} K.~Ecker  and  G.~Huisken:
A Bernstein result for minimal graphs of controlled growth. J.
Diff. Geom. {\bf 31}(1990), 397-400.

\bibitem{e-h1} K.~Ecker  and  G.~Huisken: Interior curvature
estimates for hypersurfaces of prescribed mean curvature. Ann.
Inst. H. Poincar\'e Anal. Non Lin\'eaire 6(1989), 251-260.

\bibitem{f-w} S.~Fr\"ohlich and S. Winklmann: Curvature estimates
for graphs with prescribed mean curvature and flat normal bundle.
Manuscripta math. 122(2007), 149-162.

\bibitem{h-j-w}S.~Hildebrandt, J.~Jost,~J and K.~O.~Widman:
Harmonic mappings and minimal submanifolds. Invent.math. {\bf 62}
(1980), 269-298.

\bibitem{j-x} J.~Jost and  Y.~L.~Xin:
Bernstein type theorems for higher codimension. Calculus. Var. PDE
{\bf 9} (1999), 277-296.

\bibitem{k} K.~Kenmotsu: Weierstrass formula for surfaces of
prescribed mean curvature. Math. Ann. 245(2)(1979), 89-99.

\bibitem{l-o} H. B.~Lawson and R.~Osserman:
Non-existence, non-uniqueness and irregularity of solutions to the
minimal surface system. Acta math. {\bf 139}(1977), 1-17.

\bibitem{m} J.~Moser:
On Harnack's theorem for elliptic differential equations.
 Comm. Pure Appl. Math. {\bf 14} (1961), 577-591.

\bibitem{r-v} E. A. Ruh and J. Vilms: The tension field of Gauss map.
Trans. Amer. Math. {\bf 149}(1970), 569-573.

\bibitem{s-s-y} R.~ Schoen, L.~ Simon and S.~ T.~ Yau:
Curvature estimates for minimal hypersurfaces. Acta Math. {\bf
134} (1975), 275-288.

\bibitem{jsi} J. Simons:
Minimal varieties in Riemannian manifolds Ann. Math. {\bf 88}
(1968), 62-105.

\bibitem{s-w-x}K.~Smoczyk, Guofang Wang and Y.~L.~Xin:
Bernstein type theorems with flat normal bundle. Calc. Var. and
PDE. {\bf 26(1)}(2006), 57-67.

\bibitem{x} Y. L. Xin: Geometry of harmonic maps.
Birkh\"auser PNLDE 23, (1996).

\bibitem{x1} Y.~L.~Xin:
Bernstein type theorems without graphic condition. Asia J. Math.
{\bf 9(1)}(2005), 31-44.

\bibitem{x2} Y.~L.~Xin:
Mean curvature flow with convex Gauss image. Chin. Ann. Math.
29B(2)(2008), 121-134.

\bibitem{x-y1} Y.~L.~Xin and Ling Yang:
Curvature estimates for minimal submanifolds of higher codimension.
arXiv:0709.3686.

\bibitem{x-y2} Y. L. Xin and Ling Yang: Convex functions on
Grassmannian manifolds and Lawson-Osserman problem. Adv.
Math.219(2008), 1298-1326.



\end{thebibliography}

\end{document}